\documentclass[12pt]{article}

\usepackage{amsfonts}
\usepackage{amsthm}
\usepackage[mathscr]{eucal}
\usepackage{amsmath}
\usepackage{amssymb}
\usepackage{amscd}
\usepackage{epsfig}

\theoremstyle{plain}
\newtheorem{lemma}{Lemma}[section]
\newtheorem{theorem}[lemma]{Theorem}
\newtheorem{proposition}[lemma]{Proposition}
\newtheorem{corollary}[lemma]{Corollary}

\theoremstyle{definition}
\newtheorem{example}[lemma]{Example}
\newtheorem{definition}[lemma]{Definition}
\newtheorem{remark}[lemma]{Remark}

\numberwithin{equation}{section}
\def\begeq{\stepcounter{lemma}\begin{equation}}

\date{}

\begin{document}

\title{The $H$-polynomial of a Group Embedding\footnote{This is the text of a talk given at
``A Conference in Algebraic Geometry, Honoring F. Hirzebruch 80th birthday,
May 18-23, 2008, Bar Ilan University, Israel".}}

\author{ Lex E. Renner}
\date{May 2008}
\maketitle


\begin{abstract}
The Poincar\'e polynomial of a Weyl group calculates the Betti numbers of the projective homogeneous space $G/B$, while the $h$-vector of a simple polytope calculates the Betti numbers of the corresponding rationally smooth toric variety. There is a common generalization of these two extremes called the $H$-polynomial. It applies to projective, homogeneous spaces, toric varieties and, much more generally, to any algebraic variety $X$ where there is a connected, solvable, algebraic group acting with a finite number of orbits. We illustrate this situation by describing the $H$-polynomials of certain projective $G\times G$ - varieties $X$, where $G$ is a semisimple group and $B$ is a Borel subgroup of $G$. This description is made possible  by finding an appropriate cellular decomposition for $X$ and then describing the cells combinatorially in terms of the underlying monoid of $B\times B$ - orbits. The most familiar example here is the wonderful compactification of a semisimple group of adjoint type.
\end{abstract}


\section{Introduction}   \label{intro.sec}
\vspace{10pt}
Let $G_0$ be a semisimple algebraic group  and let $\rho : G_0\to End(V)$
be a representation of $G_0$. Define $Y_\rho$ to be the Zariski closure of
$G=[\rho(G_0)]\subseteq\mathbb{P}(End(V))$, the projective space associated with $End(V)$.
Finally, let $X_\rho$ be the normalization of
$Y_\rho$. $X_\rho$ is a projective, normal  $G\times G$-embedding of $G$. That is,
there is an open embedding $G\subseteq X_\rho$ such that the action
$G\times G\times G\to G$, $(g,h,x)\leadsto gxh^{-1}$, extends over $X_\rho$.
Furthermore $B\times B$ acts on $X_\rho$ with a finite number of orbits.
See \cite{AB} for an up-to-date description of these (and other) embeddings.
The problem here is to find a homologically useful description of
how $X_\rho$ fits together from these $B\times B$-orbits. This has been accomplished by
several authors in case $X$ is the wonderful embedding of an adjoint semisimple group.
See \cite{Brion0,DP,R4,Spring} for an assortment of approaches. The purpose
of this survey is to describe what can be done here for any rationally smooth
embedding of the form $X_\rho$.

There is very little mystery in stating the problem. Suppose we have a
rationally smooth, embedding $X_\rho$. We wish to describe and
calculate the Betti numbers of $X_\rho$ in terms of the $B\times B$-orbit structure
of $X_\rho$. Our approach is based on unraveling the so called $H$-polynomial of
$X_\rho$ (see Section \ref{hpoly.sec} below). This $H$-polynomial is the
primary combinatorial invariant of a group embedding. The challenge here is to
first organize the $B\times B$-orbits into ``rational" cells, and then to quantify these cells
in terms of salient Weyl group data and toric data. The overarching fact here is
Theorem 2.5 of \cite{R8} whereby we characterize the condition ``$X_\rho$ is rationally
smooth" in terms of related toric data. This allows us to obtain a homologically meaningful
decomposition of $X_\rho$ into rational cells. The idempotent system of the associated monoid
cone $M_\rho\subseteq End(V)$ then helps us calculate the dimension of each cell.

The $H$-polynomial is depicted as a synthesis of the $h$-polynomial from toric
geometry, and the length polynomial from the theory of projective homogeneous spaces.
The purpose of the $H$-polynomial is to encode topological information about
$X_\rho$ by organizing numerical and combinatorial properties of the $B\times B$-orbits
on $X_\rho$.

Throughout this survey we omit the proofs of many statements. The interested reader
should consult \cite{R6,R7,R8,R8.5,R9} for more details.

\section{$H$-polynomials}  \label{hpoly.sec}

One can associate with a smooth torus embedding its {\bf $h$-polynomial}. And with a projective
homogeneous space (or Schubert variety) we can associate its {\bf length polynomial}.
In both cases we obtain a polynomial that can be used to calculate Betti numbers. In more general
situations, like $G\times G$-embeddings, we need to define a synthesis of these two extremes.
This leads us naturally to the notion of an {\bf $H$-polynomial}, which is the obvious synthesis
of the $h$-polynomial and the length polynomial.

\subsection{Length Polynomials}

Let $(W,S)$ be a Weyl group with length function $l : W\to\mathbb{N}$. Define
the {\bf length Polynomial} $P_W(t)$ of $W$ by setting
\[
P_W(t)=\sum_{w\in W}t^{l(w)}.
\]
It is well-known, and much-studied, that $P_W(t)$ can be used to calculate
the Betti numbers of $G/B$. This generalizes to the other projective
homogeneous spaces. If $P\subseteq G$ is a parabolic subgroup then $P=P_J$ for
some unique $J\subseteq S$. We then define the length polynomial
of $W^J$ by setting
\[
P_{W^J}(t)=\sum_{w\in W^J}t^{l(w)}.
\]
Here $W^J$ is the set of minimal coset representatives of $W_J$ in $W$.

For example consider $S_4$, the symmetric group on four letters. Its canonical
generators are the three transpositions $r=(1,2)$, $s=(2,3)$ and $t=(3,4)$. It is easy
to calculate that
\begin{enumerate}
	\item[] $|\{x\in S_4\;|\; l(x)=0\}|=1$,
	\item[] $|\{x\in S_4\;|\; l(x)=1\}|=3$,
	\item[] $|\{x\in S_4\;|\; l(x)=2\}|=5$,
	\item[] $|\{x\in S_4\;|\; l(x)=3\}|=6$,
	\item[] $|\{x\in S_4\;|\; l(x)=4\}|=5$,
	\item[] $|\{x\in S_4\;|\; l(x)=5\}|=3$ and
	\item[] $|\{x\in S_4\;|\; l(x)=6\}|=1$.
\end{enumerate}
Thus the Poincar\'e polynomial of the flag variety $\mathscr{F}_4(\mathbb{C})$ is
\[
P(t)= 1 + 3t^2 + 5t^4 + 6t^6 + 5t^8 + 3t^{10} + t^{12}.
\]
Similar formulas hold for the Grassmanians and other projective homogeneous spaces.

\subsection{$h$-polynomials} \label{hpoly.sub}

Let $X$ be a projective torus embedding. Define the {\bf $h$-polynomial} of $X$ by
setting
\[
h_X(t)=\sum_{Z<X}(t-1)^{dim(Z)}
\]
where $Z<X$ means that $Z$ is a $T$-orbit of $X$. In case $X$ is rationally smooth
(\cite{Brion1,Dan,Mc}) $h_X(t)$ can be used to calculate the Betti numbers of $X$.

For example consider the {\bf Eulerian Polynomials}. Let $S_n$ denote the symmetric group, and
let $\sigma = (p_1p_2...p_n)\in S_n$ (so that $\sigma(i)=p_i$). Define
\[
A(\sigma)=\{i\;|\; 1\leq i\leq n-1\;\text{and}\; p_i\leq p_{i+1}\},
\]
the {\bf ascent set} of $\sigma$. For example, in $S_4$,
\begin{enumerate}
\item[] $A(1234)=\{1,2,3\}$
\item[] $A(1324)=\{1,3\}$
\item[] $A(4321)=\phi$
\end{enumerate}
Define
\[
a(\sigma)=|A(\sigma)|,
\]
and
\[
E_n(t)=\sum_{\sigma\in S_n}t^{a(\sigma)}.
\]
$E_n$ is the Eulerian polynomial for $S_n$. This is the $h$-polynomial of
the torus embedding $X_{n-1}$ associated (via the normal fan construction) with the
much-studied {\bf permutahedron} $\mathscr{P}_{n-1}$. $\mathscr{P}_{n-1}$ can be defined
(ambiguously) as the convex hull of the $S_n$-orbit of a point in general position in
$\mathbb{Q}^n$. To prove that $E_n(t)$ is the $h$-polynomial of $X_{n-1}$ one can use the method
of Bialynicki-Birula \cite{BB} to obtain a cellular decomposition of $X_{n-1}$ with one cell of
dimension $a(\sigma)$ for each $\sigma\in S_n$.
The face lattice $\mathscr{F}$ of $\mathscr{P}_{n-1}$ can be identified with
\[
\bigsqcup_{\sigma\in S_n}\{(\sigma,A)\;|\; A\subseteq A(\sigma)\}.
\]
Thus the $h$-polynomial of $X_{n-1}$ is
\[
h_n(t)=\sum_{\sigma\in S_n,\ A\subseteq A(\sigma)}(t-1)^{|A|}.
\]
But
\[
t^{a(\sigma)}=((t-1)+1)^{a(\sigma)}=\sum_{A\subseteq A(\sigma)}(t-1)^{|A|}.
\]
Thus $E_n(t)=h_n(t)$.

The following illustration shows us the ascent structure for the case $S_4$.
It is a useful, but elementary exercise to calculate this polynomial by looking at the picture.
One obtains
\[
E_4(t)=1 + 11t + 11t^2 + t^3.
\]
Notice that the set of edges of the polytope $\mathscr{P}_{n-1}$ is canonically identified with
\[
\bigsqcup_{\sigma\in S_n}\{(\sigma,\alpha)\;|\; \alpha\in A(\sigma)\}.
\]
This set is also identified with the set of one-dimensional $T$-orbits on $X_{n-1}$.

\vskip-.75in
\epsfig{file=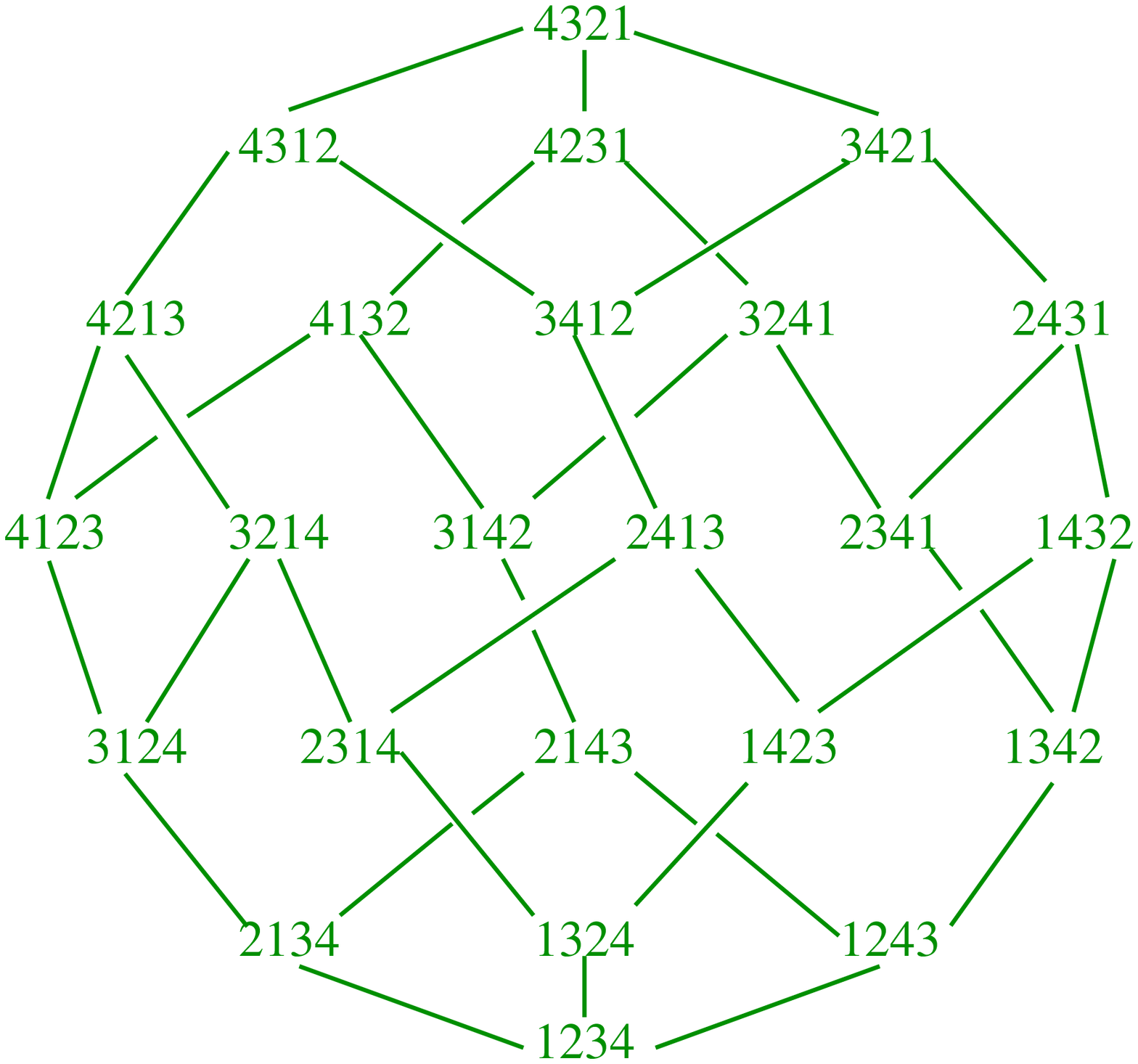, height=5.00in, bbllx=-40, bblly=100, bburx=612, bbury=792}

\subsection{$H$-polynomials}

The $H$-polynomial is the obvious synthesis of two extremes,
the $h$-polynomial of a torus embedding, and the length polynomial of a Weyl
group. In the former case one collects summands of the form $(t-1)^a$
(coming from the finite number of orbits of a torus group) while in the latter case one
collects summands of the form $t^b$ (coming from the finite number of orbits of a unipotent
group). But in each case the corresponding polynomial yields the desired coefficients. The
common theme here is that, in both cases, we are summing over a finite number of
$H$-orbits for the appropriate solvable group $H$. In more general cases, like
$G\times G$-embeddings of $G$ with the $B\times B$-action, there are a finite number of
$B\times B$-orbits, and each one is composed of a unipotent part and a diagonalizable part.
In this situation, we need to collect summands of the
form $(t-1)^at^b$ for the appropriate integers $a$ and $b$. Indeed, for each $B\times B$-orbit
$BxB$, define
\[
a(x)=rank(B\times B)-rank(B\times B)_x
\]
and
\[
b(x)=dim(UxU).
\]
Here $(B\times B)_x=\{(g,h)\in B\times B\;|\; gxh^{-1}=x\}$.
Thus we make the following fundamental definition.

\begin{definition} \label{hpolynomial.def}
Let $\rho : G\to End(V)$ be an irreducible representation and let
\[
X=X_\rho
\]
be as above. The {\bf $H$-polynomial} of $X$ is defined to be
\[
H_X(t)=\sum_{x\in\mathscr{R}}(t-1)^{a(x)}t^{b(x)}.
\]
where $\mathscr{R}$ is a set of representatives for the $B\times B$-orbits of $X$.
\end{definition}

\begin{remark}
This $H$-polynomial is not the correct tool for investigating varieties with singularities
that are not rationally smooth. In the case of Schubert varieties, and Kazhdan-Lusztig theory,
the correct formulation incorporates a ``correction factor" ({\em aka} the {\bf $KL$-polynomial})
that takes into account local intersection cohomology groups. See Theorem 6.2.10 of
\cite{BL}.

The authors of \cite{BJ} calculate the Poincar\'e polynomial, for intersection cohomology,
of a large class of  $G\times G$-embeddings using the stratification by $G\times G$-orbits.
In case the singularities of $\mathbb{P}(M)$ are rationally smooth, the polynomial $IP_X(t)$
of \cite{BJ} agrees with the polynomial $H_X(t)$ (where $X=\mathbb{P}(M)$) defined above in
Definition \ref{hpolynomial.def}. See Theorem \ref{hpolyequalspoincarepoly.thm}.
However, in the absence of rationally smooth singularities, these local intersection cohomology
groups may not be so well adapted to cellular decompositions.

See \cite{Spring} for a detailed study of the intersection cohomology of $B\times B$-orbit closures
in the case of the ``wonderful embedding" (i.e. When $M$ is $\mathscr{J}$-irreducible of type
$J=\emptyset$ and $X=(M\setminus\{0\})/\mathbb{C}^*$. Even though $M\setminus\{0\}$ is rationally
smooth in this case, the same may not true for the closure in $M\setminus\{0\}$, of a $B\times B$-orbit).
\end{remark}

\begin{example}   \label{easyhpoly.eg}
Let $M=M_2(K)$. Then
\[ \mathscr{R}=\left\{\left(\begin{array}{ll}
1 & 0\\
0 & 1
\end{array}\right) , \left(\begin{array}{ll}
0 & 1\\
1 & 0
\end{array}\right) , \left(\begin{array}{ll}
1 & 0\\
0 & 0
\end{array}\right) , \left(\begin{array}{ll}
0 & 0\\
0 & 1
\end{array}\right) , \left(\begin{array}{ll}
0 & 1\\
0 & 0
\end{array}\right) , \left(\begin{array}{ll}
0 & 0\\
1 & 0
\end{array}\right) , \left(\begin{array}{ll}
0 & 0\\
0 & 0
\end{array}\right)\right\}\; . \]
For simplicity we write it, in order, as
$\mathscr{R}=\{1,s,e,f,n,m,0\}$. It is easy to calculate $a(x)$ and $b(x)$
for each element $x\in\mathscr{R}$. For example, $a(m)=1$ and $b(m)=2$. Thus,
if $X=M_2(\mathbb{C})$,
\[
H_X(t)=(t-1)^2t^1+(t-1)^2t^2+(t-1)^1t^1+(t-1)^1t^1+(t-1)^1t^0+(t-1)^1t^2+(t-1)^0t^0
\]
Hence, after an elementary calculation, we obtain
\[
H_X(t)=t^4.
\]
This should not be surprising since, over $\mathbb{F}_q$,
\[
|BxB|=(q-1)^{a(x)}q^{b(x)}.
\]
while
\[
M_2(\mathbb{C})=\bigsqcup_{x\in\mathscr{R}}BxB.
\]
Here, $B$ is the $2\times 2$ upper-triangular group.
\end{example}

\begin{example}  \label{wonderfulpgl3.ex}
Let $G_0=PGL_3(\mathbb{C})$, and let $\rho : G_0\to End(V)$
be any irreducible representation whose highest weight is in general
position. Then the $H$-polynomial of $X_\rho$ is given by
\[
H(t)=\left[1+2t^2+2t^3+t^5\right]\left[1+2t+2t^2+t^3\right]
\]
See \S\ref{wonderful.sub} and Theorem \ref{thebiganswer.thm} for some related
 general formulas.
\end{example}

\section{Rationally Smooth Embeddings}  \label{ratsm.sec}

Let $X$ be a complex, algebraic variety of dimension $n$.
Then $X$ is {\bf rationally smooth} at $x\in X$ if there is a neighborhood
$U$ of $x$ in the complex topology such that, for any $y\in U$,
\[
H^m(X,X\setminus\{y\})=(0)
\]
for $m\neq 2n$ and
\[
H^{2n}(X,X\setminus\{y\})=\mathbb{Q}.
\]
Here $H^*(X)$ denotes the cohomology of $X$ with rational coefficients.
Danilov, in \cite{Dan}, characterized the rationally smooth toric varieties
in combinatorial terms.

The purpose of this section is to identify the class of rationally smooth
embeddings of the form $X_\rho$. We then find a useful description of the
$H$-polynomial of these embeddings. See Theorem\ref{thebiganswer.thm}.

\subsection{Rationally Smooth Monoids} \label{ratsmmon.sub}

Let $M$ and $N$ be reductive monoids. We write $M\sim_0 N$ if there is a reductive
monoid $L$ and finite dominant morphisms $L\to M$ and $L\to N$ of algebraic monoids.
One can check that this is indeed an equivalence relation.

In the following theorem we write $\mathbb{C}^n$ for the reductive monoid with
multiplication
\[
m : \mathbb{C}^n\times \mathbb{C}^n \to\mathbb{C}^n
\]
defined by $m((t_1,...t_n),(x_1,...x_n)) = (t_1x_1,...t_nx_n)$.

\begin{theorem}  \label{mainratsm.thm}
Let $M$ be a reductive monoid with zero. The following are equivalent.
\begin{enumerate}
\item $M$ is rationally smooth.
\item $M\sim_0 \Pi_iM_{n_i}(\mathbb{C})$.
\item $\overline{T}\sim_0 \mathbb{C}^m$.
\end{enumerate}
\end{theorem}

Notice that if $M\sim_0 N$ then $M$ is rationally smooth if and only if $N$
is rationally smooth.

See \cite{R8} for a proof of Theorem \ref{mainratsm.thm}. See also \cite{Brion1} for a
systematic discussion of rationally smooth varieties with torus action.

\begin{corollary} \label{mainratsm.cor}
Let $M$ be a rationally smooth reductive monoid and let $e\in E(M)$. Then
$eM$ is a rationally smooth algebraic variety.
\end{corollary}
\begin{proof}
Let $T$ be a maximal torus with $e\in\overline{T}$. Then there is a one-parameter
subgroup $S\subseteq T$ such that $E(\overline{S})=\{1,e\}$. One checks easily that
$eM=\{x\in M\;|\; sx=x\;\text{for all}\;s\in S\}.$ Thus, by Theorem 1.1 of \cite{Brion1},
$eM$ is rationally smooth.
\end{proof}

In \cite{R6} it is shown that any rationally smooth embedding of the form $X_\rho$
(see \S\ref{ratsmemb.sub} below) has a natural cell decomposition. Corollary \ref{mainratsm.cor}
above allows one to conclude that these cells are themselves rationally smooth. See Theorem 5.1
of \cite{R6}.

\subsection{Rationally Smooth Embeddings} \label{ratsmemb.sub}

Let $M$ be a normal, reductive monoid with unit group $G$ and zero element
$0\in M$. Let $\epsilon : \mathbb{C}^*\to G$ be a central 1-parameter subgroup
that converges to 0. Define
\[
\mathbb{P}_{\epsilon}(M) = (M\setminus\{0\})/\mathbb{C}^*.
\]
$\mathbb{P}_{\epsilon}(M)$ is a projective $G\times G$-embedding.
$G\times G$ acts on $\mathbb{P}_{\epsilon}(M)$ by
\[
(g,h,[x])\leadsto [gxh^{-1}].
\]
If $M$ is semisimple (so that $\epsilon$ is essentially unique)
we write $\mathbb{P}(M)$ for $\mathbb{P}_{\epsilon}(M)$. Notice also
that there is a representation $\rho$ of $G$
such that
\[
\mathbb{P}_{\epsilon}(M)=X_\rho.
\]

Recall that $E(M)=\{e\in M\;|\; e^2 = e\}$. For $e\in E(M)$ we define
\[
M_e = \overline{\{g\in G\;|\; ge=eg=e \}}.
\]
By the results of \cite{Brion2}, $M_e$ is an irreducible, normal reductive monoid
with unit group $G_e=\{g\in G\;|\; ge=eg=e \}$.

\begin{theorem} \label{ratsmemb.thm}
The following are equivalent.
\begin{enumerate}
	\item $\mathbb{P}_{\epsilon}(M)$ is rationally smooth.
	\item For any $e\in E(M)\setminus\{0\}$, $M_e$ is rationally
	smooth.
\end{enumerate}
\end{theorem}

Indeed, by the results of \cite{Brion2,R2}, $M\setminus\{0\}$ is locally isomorphic to
the product of $M_e$ ($e\in E_1$) with some affine space. Thus
$\mathbb{P}_{\epsilon}(M)$ is rationally smooth if and only if so is
$M_e$ for all $e\in E_1$.

Notice, in particular, that the condition ``$\mathbb{P}_{\epsilon}(M)$ is rationally smooth"
is independent of how we choose $\epsilon$.

\subsection{H-polynomials and Poincar\'e Polynomials}

\begin{definition}  \label{poincarepoly.def}
Let $X$ be a complex algebraic variety. The {\bf Poincar\'e polynomial}
of $X$ is the polynomial $P_X(t)$ with the signed Betti numbers of
$X$ as coefficients.
\[
P_X(t) = \sum_{i\geq 0}(-1)^idim_\mathbb{Q}[H^i(X;\mathbb{Q})]t^i.
\]
\end{definition}

Clearly one can define a Poincar\'e polynomial for any reasonable cohomology theory.
In \cite{BJ} the authors compute the {\bf intersection cohomology}
Poincar\'e polynomial $IP_X(t)$ for a large class of $G\times G$-embeddings
$X$ of $G$. However it is known (\cite{Mc}) that $IP_X(t)=P_X(t)$ in case $X$ has rationally
smooth singularities.

\begin{theorem} \label{hpolyequalspoincarepoly.thm}
Let $M$ be a semisimple algebraic monoid such that $M\setminus\{0\}$ is rationally smooth.
Then
\[
H_X(t^2)=P_X(t)
\]
where $X=[M\setminus\{0\}]/K^*$.
\end{theorem}
\begin{proof}
We give a sketch. See \cite{R9} for more details.
By our assumptions on $M$, $X$ is rationally smooth. Hence by the results of McCrory in \cite{Mc},
$IP_X(t)=P_X(t)$. Hence it suffices to show that $H(M)(t^2)=IP_X(t)$.

Let $x\in X$. Then from Theorem 1.1 of \cite{BJ}
\[
IP_{X,x}(t)=\tau_{\leq d_x-1}((1-t^2)IP_{\mathbb{P}(S_x)}(t))
\]
where $S_x$ is the appropriate slice and $d_x=dim(S_x)$. One checks, using
Lemma 1.3 of  \cite{Brion1} and the results of
\cite{Brion2} , that $\mathbb{P}(S_x)$ is a rational homology projective space of dimension $d_x-1$.
Thus $IP_{X,x}(t)=1$. Consequently, the formula for $IP_X(t)$ in Theorem 1.1 of \cite{BJ}
simplifies to a summation with summands of the form $P_{(G\times G)x}(t)$, as in (5.1.5) of \cite{BJ}.
Hence
\[
IP_X(t)=\sum_xP_{(G\times G)x}(t),
\]
where the sum is taken over a set of representatives of the $G\times G$-orbits of $X$. But this
is the same formula that one obtains by combining the $B\times B$-orbits into one summand for
each $G\times G$-orbit in the formula for $H(M)(t^2)$.
\end{proof}

\section{Cellular Decompositions and $H$-polynomials}

Assume that $X_\rho$ is rationally smooth. We then obtain a kind of cellular
decomposition for $X_\rho$. This decomposition is ultimately controlled
by the idempotent system of the associated monoid cone $M$ of $X_\rho$.
See Theorem \ref{monoidbb.thm} below. This allows us to find a
useful way to label these cells in terms of quantities from $W$ and
$E(\overline{T})$. We then find a dimension formula for each cell in
terms of its label. Finally we find an appropriate homological interpretation
of these cells so that we can compute the Betti numbers of $X_\rho$ in terms
of the $H$-polynomial. See also Theorem \ref{hpolyequalspoincarepoly.thm}.

\subsection{Monoid BB-decompositions}  \label{monoidbb.sub}

In this section we study the $BB$-cells of
$\mathbb{P}_{\epsilon}(M)=(M\backslash\{0\})/Z$
where $M$ is a reductive monoid with zero and
$\epsilon : Z\subset G$ is a central
one-parameter subgroup with $0\in\overline{Z}$.

Let $X$ be a normal, projective variety and assume that
$S=K^*$ acts on $X$. If $F_i\subset X^S$ is a connected component of the
fixed point set $X^S$ we define, following \cite{BB},
\[
X_i=\{x\in X\;|\; \underset{t\to 0}{lim}(t\cdot x)\in F_i\}.
\]
This decomposes $X$ as a disjoint union
\[
X=\bigsqcup_i X_i
\]
of locally closed subsets. Furthermore we have the $BB$-maps
\[
\pi_i : X_i\to F_i
\]
defined by $\pi_i(x)=\underset{t\to 0}{lim}(t\cdot x)$. See \cite{BB} for more
details. In that paper the author assumes that $X$ is nonsingular. Then he proves
his much-celebrated results (see Theorem 4.3 of \cite{BB}).
However many of his ideas can be extended to the nonsingular case. On the other hand,
the purpose of our discussion is to describe this {\bf $BB$-decomposition} in terms
of the system of idempotents of an appropriate algebraic monoid.

We start with the following results from \cite{R6}.
These theorems are stated in \cite{R6} for semisimple
monoids, but the proofs are easily modified to apply
to all reductive monoids with zero. Let $\overline{T}$ be the
closure in $M$ of a maximal torus, and let $E(\overline{T})$
be the set of idempotents of $\overline{T}$. Let
\[
E_1=E_1(\overline{T})=\{e\in\overline{T}\;|\; dim(Te)=1\},
\]
the set of rank-one idempotents of $\overline{T}$.\\

\begin{theorem}   \label{monoidbb.thm}
Let $M$ be a reductive monoid with unit
group $G$, zero element $0\in M$,
Borel subgroup $B\subseteq G$ and maximal torus $T\subseteq B$.
Let $\epsilon : Z\subseteq G$ be a 1-dimensional, connected, central
subgroup of $G$ such that $0\in\overline{Z}$.
Choose a one parameter subgroup $\lambda : \mathbb{C}^*\to T$ such that
\begin{enumerate}
\item $\underset{t\to 0}{lim}(tut^{-1})=1$ for all $u\in B_u$.
\item $\{x\in M\backslash\{0\}\;|\; \lambda(t)x\in Zx\;\text{for
all}\; t\in\mathbb{C}^*\}=\bigcup_{e\in E_1(\overline{T})}eM$.
\end{enumerate}
Let $X=(M\backslash\{0\})/Z=\mathbb{P}_\epsilon(M)$ and let
\[
X=\bigsqcup_{e\in E_1}X(e)
\]
be the $BB$-decomposition of $X$ with respect to $\lambda$.
Then, for all $e\in E_1(\overline{T})$,
\[
X(e)\subseteq\{[y]\in X\;|\; eBy=eBey\subseteq eG\}.
\]
Furthermore, the $BB$ projection,
\[
\pi_e : X(e)\to (eG)/\mathbb{C}^* \cong G/P_e
\]
is given by $\pi_e([y])=[ey]$.
\end{theorem}

Theorem \ref{monoidbb.thm} says that, if we choose an appropriate one-parameter-subgroup
of $T$,  the $BB$-projections
can be calculated by using the rank-one idempotents of $\overline{T}$. Since we
assume that $\mathbb{P}_{\epsilon}(M)=X_\rho$ is rationally smooth, the fibre of each
$\pi_e$ is a certain ``rational cell" that has been computed in \cite{R6}. On the
other hand, $G/P_e$ has an easily quantified cell decomposition
\[
G/P_e = \sqcup A_w
\]
Thus, if we let $\pi_e^{-1}(A_w)=C_w$, we obtain the decomposition $X_e=\sqcup_w C_w$.
From Corollary 5.2 of \cite{R6} each cell $C$ has an $H$-polynomial of the form
\[
H(C)=\sum_{r\in\mathscr{R}\cap C}t^{a(r)}(t-1)^{b(r)}=t^{dim(C)}.
\]
Putting this all together, we can calculate the $H$-polynomial of $\mathbb{P}_\epsilon(M)$
by identifying the contribution of each $C$ and each $X(e)$.

To state the main theorem we first recall that if $(W,S)$ is a Weyl group and $J\subset S$
then $W^J$ is the set of minimal length representatives for the cosets of $W_J$ in $W$.
In particular, the canonical composition
\[
W^J\to W\to W/W_J
\]
is bijective.

\begin{theorem}  \label{thebiganswer.thm}
Let $M$ be a reductive monoid such that $\mathbb{P}_\epsilon(M)$ is rationally smooth, and let
$\mathscr{R}$ be the
monoid of $B\times B$-orbits of $M$. Write $G/P_e$ for $(eG)/\mathbb{C}^*$. Let $e_1\in\Lambda_1$
and let $we_1w^{-1}=e$, where $w\in W^J$.
$w_0\in W^J$ is the unique element of maximal length, and $J=\{s\in S\;|\; se_1=e_1s\}$.
Also, let $H(G/P_e)=\sum_{w\in W^J}t^{l(w)}$.
Finally, $\nu(e)$ is defined in Definition 5.4 of \cite{R6}.
\begin{enumerate}
\item  If we let $w(e)=w$ then
\[
H_{\mathbb{P}_\epsilon(M)}(t)=\sum_{e\in E_1}\left[t^{l(w_0)-l(w(e))+\nu(e)}H(G/P_e)\right].
\]
\item In case $P_e$ and $P_{e'}$ are conjugate for all $e,e'\in E_1$ the sum can be
rewritten as
\[
H_{\mathbb{P}_\epsilon(M)}(t)=\left[\sum_{e\in E_1}t^{l(w_0)-l(w(e))+\nu(e)}\right]H(G/P).
\]
where $P=P_e$.
\end{enumerate}
\end{theorem}
See Theorem 5.5 of \cite{R6}.

\begin{remark} \label{itsnu.rk}
$\nu(e)$ is the most difficult quantity to calculate in the above setup.
It contains a subtle contribution from the induced $BB$-decomposition of the
associated maximal torus. See Sections 4.2 and 5.1 of \cite{R6} for more details.
In some examples this calculation involves {\bf descent systems} \cite{R7}. See Section
\ref{descentsys.sec} for a discussion of these descent systems, and
see Section \ref{wonderful.sub} below (the wonderful embedding) for the motivating example.
See Theorem \ref{thejewel.thm} and Examples \ref{allbutone.ex}, \ref{allbuttwo.ex} and \ref{coolcube.ex}
for more illustration of how $\nu(e)$ is quantified in the case of a {\bf simple} embedding
of the form $X_\rho$.

Each summand in the formula for $H_{\mathbb{P}(M)}(t)$ (in Theorem \ref{thebiganswer.thm}) mirrors
the $BB$-projection
\[
\pi_e : X(e)\to G/P_e.
\]
In particular, the fibre of $\pi_e$ has dimension $l(w_0)-l(w(e))+\nu(e)$.
\end{remark}

\section{Descent Systems} \label{descentsys.sec}

As mentioned above in Remark \ref{itsnu.rk}, a major problem in calculating the $H$-polynomial
is finding a satisfying description of the quantity $\nu(e)$ of Theorem \ref{thebiganswer.thm}.
We are particularly interested in the situation where the embedding $X_\rho$ is obtained
from an irreducible representation $\rho=\rho_\lambda$, of type $J=\{s\in S\;|\; s(\lambda)=\lambda\}$,
of $G$.  We refer to this embedding as $\mathbb{P}(J)$.
This is well-defined since $\mathbb{P}(J)$ depends only on $J$.
This leads us to the notion of a {\bf descent system} \cite{R7}. This descent system serves as an
effective combinatorial substitute for the infinitesimal part of Bialynicki-Birula's method \cite{BB}.

\begin{definition}\label{descentset.def}
Let $(W,S)$ be a Weyl group and let $J\subseteq S$ be a proper subset.
Define
\[
S^J=(W_J(S\setminus J)W_J)\cap W^J.
\]
We refer to $(W^J,S^J)$ as the {\bf descent system} associated with
$J\subseteq S$.
\end{definition}

\begin{proposition} \label{dichotomy.prop}
Let $u,v\in W^J$ be such that $u^{-1}v\in S^JW_J$. In
particular, $u\neq v$. Then either
$u<v$ or $v<u$ in the Bruhat order $<$ on $W^J$.
\end{proposition}

For a proof see \cite{R7}.

We let
\[
S^J_s=W_JsW_J\cap W^J.
\]

\begin{definition} \label{descentsets.def}
Let $w\in W^J$. Define
\begin{enumerate}
\item $D^J_s(w)=\{r\in S^J_s\;|\;wrc<w\;\text{for some}\;c\in W_J\}$, and
\item $A^J_s(w)=\{r\in S^J_s\;|\;w<wr\}$.
\end{enumerate}
We refer to $D^J(w)=\bigsqcup_{s\in S\setminus J}D^J_s(w)$ as the {\bf descent set} of $w$
relative to $J$, and $A^J(w)=\sqcup_{s\in S\setminus J}A^J_s(w)$ as the {\bf ascent set}
of $w$ relative to $J$.
\end{definition}

By Proposition~\ref{dichotomy.prop}, for any $w\in W^J$, $S^J=D^J(w)\sqcup
A^J(w)$.

\begin{remark}  \label{minlength.rk}
Notice that $wrc<w$ for some $c\in W_J$ if and only if $(wr)_0<w$, where
$(wr)_0\in wrW_J$ is the element of minimal length in $wrW_J$. It is
useful to illustrate the fact that $S^J=D^J(w)\sqcup A^J(w)$, for
each $w\in W^J$, by doing some specific calculations.
\end{remark}

\begin{definition}   \label{augmented.def}
For each $w\in W^J$ and each $s\in S\setminus J$ define
$\nu_s(w)=|A^J_s(w)|$. We refer to $(W^J,\leq,\{\nu_s\})$ as the
{\bf augmented poset} of $J$. For convenience we let
\[
\nu(w)=\sum_{s\in S\setminus J}\nu_s(w).
\]
\end{definition}
The point here is this. We use $(W^J,\leq,\{\nu_s\})$ to quantify how
the underlying torus embedding of $\mathbb{P}(J)$ is involved in calculating the
$H$-polynomial of $\mathbb{P}(J)$.

\begin{example} \label{thategfromwgspbppp.ex}
Let
\[
W=<s_1,...s_n>
\]
be the Weyl group of type $A_n$ (so that $W\cong S_{n+1}$),
and let
\[
J=\{s_3,...s_n\}\subseteq S.
\]

If $w\in W^J$ then $w=a_p$, $w=b_q$, or else  $w=a_pb_q$.
Here $a_p=s_p\cdots s_1$ ($1\leq p\leq n$) and $b_q=s_q\cdots s_2$ ($2\leq q\leq n$).
If we adopt the useful convention $a_0=1$ and $b_1=1$, then we can write
\[
W^J=\{a_pb_q\;|\;0\leq p\leq n\;\text{and}\;1\leq q\leq n\}
\]
with uniqueness of decomposition. Let $w=a_pb_q\in W^J$. After some tedious
calculation with braid relations and reflections, we obtain that
\begin{enumerate}
	\item[a)] $A^J_{s_1}(a_pb_q)=\{s_1\}$ if $p < q$.\\
	      $A^J_{s_1}(a_pb_q)=\emptyset$ if $q \leq p$.\\
	      Thus $\nu_{s_1}(a_pb_q)=1$ if $p < q$ and $\nu_{s_1}(a_pb_q)=0$ if $q \leq p$.
	\item[b)] $A^J_{s_2}(a_pb_q)=\{s_m\cdots s_n\;|\; m>q\}$ if $q<n$.\\
	      $A^J_{s_2}(a_pb_q)=\emptyset$ if $q=n$.\\
	      Thus $\nu_{s_2}(a_pb_q)=n-q$.
\end{enumerate}
It is interesting to compute the two-parameter ``Euler polynomial"
\[
H(t_1,t_2)=\sum_{w\in W^J}t_1^{\nu_1(w)}t_2^{\nu_2(w)}
\]
of the augmented
poset $(W^J,\leq,\{\nu_1,\nu_2\})$ (where we write $\nu_i$ for $\nu_{s_i}$). A
simple calculation yields
\[
H(t_1,t_2)=\sum_{k=1}^n[kt_1+(n+1-k)]t_2^{n-k}.
\]
\end{example}

Let $r : W\to GL(V)$ be the usual reflection representation of the
Weyl group $W$, where $V$ is a rational vector space. Along with this goes the
{\bf Weyl chamber} $\mathscr{C}\subseteq V$ and the corresponding set of {\bf simple
reflections} $S\subseteq W$. The Weyl group $W$ is generated by $S$, and $\mathscr{C}$ is
a fundamental domain for the action of $W$ on $V$.

Let $\lambda\in\mathscr{C}$. Consider the face lattice $\mathscr{F}_\lambda$
of the polytope
\[
\mathscr{P}_\lambda=Conv(W\cdot\lambda),
\]
the convex hull of $W\cdot\lambda$
in $V$. This lattice $\mathscr{F}_\lambda$ depends only on $W_\lambda=\{w\in W\;|\;
w(\lambda)=\lambda\}=W_{J}=\langle s\;|\;s\in J\rangle$, where $J=\{s\in S\;|\; s(\lambda)=\lambda\}$.
One can describe $\mathscr{F}_\lambda=\mathscr{F}_J$ explicitly in terms of $J\subseteq S$. See
\cite{PR1}.

\begin{definition}   \label{combsmooth.def}
We refer to $J$ as {\bf combinatorially smooth} if
$\mathscr{P}_\lambda$ is a simple polytope.
\end{definition}

It is important to characterize the very interesting condition of Definition \ref{combsmooth.def}.
If $J\subseteq S$ we let $\pi_0(J)$ denote the set of connected components of $J$.
To be more precise, let $s,t\in J$. Then $s$ and $t$ are in the same connected component
of $J$ if there exist $s_1,...,s_k\in J$ such that $ss_1\neq s_1s$, $s_1s_2\neq s_2s_1$,....,
$s_{k-1}s_k\neq s_ks_{k-1}$, and $s_kt\neq ts_k$.

The following theorem indicates exactly how to detect, combinatorially, the
condition of Definiton~\ref{combsmooth.def}.

\begin{theorem}     \label{almostsmooth.thm}
Let $\lambda\in\mathscr{C}$. The following are equivalent.
\begin{enumerate}
\item $\mathscr{P}_\lambda$ is a simple polytope.
\item There are exactly $|S|$ edges of $\mathscr{P}_\lambda$ meeting at $\lambda$.
\item $J=\{s\in S\;|\; s(\lambda)=\lambda\}$ has the properties
  \begin{enumerate}
    \item If $s\in S\backslash J$, and $J\not\subseteq C_W(s)$, then there is a unique
    $t\in J$ such that $st\neq ts$. If $C\in\pi_0(J)$ is the unique
    connected component of $J$ with $t\in C$ then $C\backslash\{t\}\subseteq C$
    is a setup of type $A_{l-1}\subseteq A_l$.
    \item For each $C\in\pi_0(J)$ there is a unique $s\in S\backslash J$
    such that $st\neq ts$ for some $t\in C$.
  \end{enumerate}
\end{enumerate}
\end{theorem}

One can list all possible subsets $J\subseteq S$ that are combinatorially smooth.
We do this according to the type of the underlying simple group.
The numbering of the elements of $S$ is as follows. For types $A_n, B_n, C_n, F_4,$ and
$G_2$ it is the usual numbering. In these cases the end nodes are $s_1$ and $s_n$. For type
$E_6$ the end nodes are $s_1,s_5$ and $s_6$ with $s_3s_6\neq s_6s_3$. For type
$E_7$ the end nodes are $s_1,s_6$ and $s_7$ with $s_4s_7\neq s_7s_4$. For type
$E_8$ the end nodes are $s_1,s_7$ and $s_8$ with $s_5s_8\neq s_8s_5$.
In each case  of type $E_n$, the nodes corresponding to $s_1, s_2,...,s_{n-1}$
determine the unique subdiagram of type $A_{n-1}$. For type $D_n$ the end
nodes are $s_1,s_{n-1}$ and $s_n$. The two subdiagrams of $D_n$, of type $A_{n-1}$,
correspond to the subsets $\{s_1, s_2,...,s_{n-2},s_{n-1}\}$ and
$\{s_1, s_2,...,s_{n-2},s_n\}$ of $S$.

\begin{enumerate}
  \item $A_1$.
  \begin{enumerate}
	\item $J=\phi$.
  \end{enumerate}
        $A_n$, $n\geq 2$. Let $S=\{s_1,...,s_n\}$.
\begin{enumerate}
  \item $J=\phi$.
	\item $J=\{s_1,...,s_i\}$, $1\leq i<n$.
	\item $J=\{s_j,...,s_n\}$, $1<j\leq n$.
	\item $J=\{s_1,...,s_i,s_j,...s_n\}$, $1\leq i$, $i\leq j-3$ and $j\leq n$.
\end{enumerate}
	\item $B_2$.
	\begin{enumerate}
	\item $J=\phi$.
	\item $J=\{s_1\}$.
	\item $J=\{s_2\}$.
\end{enumerate}
	      $B_n$, $n\geq 3$. Let $S=\{s_1,...,s_n\}$, $\alpha_n$ short.
\begin{enumerate}
  \item $J=\phi$.
	\item $J=\{s_1,...,s_i\}$, $1\leq i<n$.
	\item $J=\{s_n\}$.
	\item $J=\{s_1,...,s_i,s_n\}$, $1\leq i$ and $i\leq n-3$.
\end{enumerate}
   \item $C_n$, $n\geq 3$. Let $S=\{s_1,...,s_n\}$, $\alpha_n$ long.
\begin{enumerate}
  \item $J=\phi$.
	\item $J=\{s_1,...,s_i\}$, $1\leq i<n$.
	\item $J=\{s_n\}$.
	\item $J=\{s_1,...,s_i,s_n\}$, $1\leq i$ and $i\leq n-3$.
\end{enumerate}
	\item $D_n$, $n\geq 4$. Let $S=\{s_1,...s_{n-2},s_{n-1},s_n\}$.
\begin{enumerate}
	\item $J=\phi$.
	\item $J=\{s_1,...,s_i\}$, $1\leq i\leq n-3$.
	\item $J=\{s_{n-1}\}$.
	\item $J=\{s_n\}$.
	\item $J=\{s_1,...,s_i,s_{n-1}\}$, $1\leq i\leq n-4$.
	\item $J=\{s_1,...,s_i,s_n\}$, $1\leq i\leq n-4$.
\end{enumerate}
	\item $E_6$. Let $S=\{s_1,s_2,s_3,s_4,s_5,s_6\}$.
\begin{enumerate}
	\item $J=\phi$.
	\item $J=\{s_1\}$ or $\{s_1,s_2\}$.
	\item $J=\{s_5\}$ or $\{s_4,s_5\}$.
	\item $J=\{s_6\}$.
  \item $J=\{s_1,s_5\},\{s_1,s_2,s_5\}$ or $\{s_1,s_4,s_5\}$.
  \item $J=\{s_1,s_6\}$.
  \item $J=\{s_5,s_6\}$
  \item $J=\{s_1,s_5,s_6\}$.
\end{enumerate}
	\item $E_7$. Let $S=\{s_1,s_2,s_3,s_4,s_5,s_6,s_7\}$.
\begin{enumerate}
  \item $J=\phi$.
	\item $J=\{s_1\}, \{s_1,s_2\}$ or $\{s_1,s_2,s_3\}$.
	\item $J=\{s_6\}$ or $\{s_5,s_6\}$.
	\item $J=\{s_7\}$.
  \item $J=\{s_1,s_6\},\{s_1,s_2,s_6\},\{s_1,s_2,s_3,s_6\},\{s_1,s_5,s_6\},$
  or $\{s_1,s_2,s_5,s_6\}$.
  \item $J=\{s_6,s_7\}$.
  \item $J=\{s_1,s_7\}$ or $\{s_1,s_2,s_7\}$.
  \item $J=\{s_1,s_6,s_7\},\{s_1,s_2,s_6,s_7\}$.
\end{enumerate}
	\item $E_8$. Let $S=\{s_1,s_2,s_3,s_4,s_5,s_6,s_7,s_8\}$.
\begin{enumerate}
  \item $J=\phi$.
	\item $J=\{s_1\}, \{s_1,s_2\}, \{s_1,s_2,s_3\}$ or $\{s_1,s_2,s_3,s_4\}$.
	\item $J=\{s_7\}$ or $\{s_6,s_7\}$.
	\item $J=\{s_8\}$.
  \item $J=\{s_1,s_7\},\{s_1,s_2,s_7\},\{s_1,s_2,s_3,s_7\}, \{s_1,s_2,s_3,s_4,s_7\}$,\\
  $\{s_1,s_6,s_7\},\{s_1,s_2,s_6,s_7\}, \{s_1,s_2,s_3,s_6,s_7\}$
  or $\{s_1,s_2,s_5,s_6\}$.
  \item $J=\{s_7,s_8\}$.
  \item $J=\{s_1,s_8\}, \{s_1,s_2,s_8\}$ or $\{s_1,s_2,s_3,s_8\}$.
  \item $J=\{s_1,s_7,s_8\},\{s_1,s_2,s_7,s_8\}$.
\end{enumerate}
	\item $F_4$. Let $S=\{s_1,s_2,s_3,s_4\}$.
\begin{enumerate}
	\item $J=\phi$.
	\item $J=\{s_1\}$ or $\{s_1,s_2\}$.
	\item $J=\{s_4\}$ or $\{s_3,s_4\}$.
	\item $J=\{s_1,s_4\}$.
\end{enumerate}
	\item $G_2$. Let $S=\{s_1,s_2\}$.
\begin{enumerate}
	\item $J=\phi$.
	\item $J=\{s_1\}$.
	\item $J=\{s_2\}$.
\end{enumerate}
\end{enumerate}

See Corollary 3.5 of \cite{R7} for more discussion.

Associated with $J\subseteq S$ is the rational polytope $\mathscr{P}_\lambda$,
where $\lambda$ is chosen so that $J=\{s\in S\;|\; s(\lambda)=\lambda\}$.
This polytope determines a projective torus embedding $X(J)$ via the ``central
fan construction", or more naively as
\[
X(J)=[\overline{\rho_\lambda(T)}\setminus\{0\}]/K^*,
\]
where $\rho_\lambda$ is the irreducible representation of $G$ with highest weight
$\lambda$. $X(J)$ is independent of $\lambda$ and depends only on $J$.
The following theorem vindicates the introduction of descent systems.

\begin{theorem} \label{ratsmsimpleemb2.thm}
The following are equivalent.
\begin{enumerate}
	\item $J$ is combinatorially smooth.
	\item $X(J)$ is rationally smooth.
	\item $\mathbb{P}(M)$ is rationally smooth.
\end{enumerate}
\end{theorem}

Indeed, Theorem \ref{ratsmsimpleemb2.thm} follows directly from Theorem
\ref{mainratsm.thm} and Theorem \ref{ratsmemb.thm}.

These toric varieties $X(J)$ are of interest in there own right. The following
Theorem is proved in \cite{R8.5}.

\begin{theorem}   \label{bettinumber.thm}
Assume that $X(J)$ is rationally  smooth. Then the Poincar\'e polynomial
of $X(J)$ is
\[
P(X(J),t)=\sum_{w\in W^J}t^{2\nu(w)}.
\]
where $\nu$ is as in Definition \ref{augmented.def}.
\end{theorem}

See Examples 4.6 and 4.9 of \cite{R8.5} for the details of the following two examples.

\begin{example}
Let $(W_n,S_n)=<s_1,s_2,...,s_n>$ ($n\geq 2$), where
$S_n=\{s_1,s_2,... ,s_n\}$ (the symmetric group) and
let $J=\{s_3,s_4,...,s_n\}\subset S_n$. Then $X(J)=X_n(J)$ is rationally smooth
and
\[
P(X_n(J),t)=t^{2n}+(n+2)t^{2(n-1)}+(n+2)t^{2(n-2)}+...+(n+2)t^4+(n+2)t^2+1.
\]
\end{example}

\begin{example}
In this example we consider the root system of type $B_l$. Let $E$ be a real vector
space with orthonormal basis $\{\epsilon_1,...,\epsilon_l\}$. Then
\begin{enumerate}
\item[] $\Phi^+=\{\epsilon_i-\epsilon_j\;|\; i<j\}\cup\{\epsilon_i+\epsilon_j\;|\;
        i\neq j\}\cup\{\epsilon_i\}$, and
\item[] $\Delta =\{\epsilon_1-\epsilon_2,...,\epsilon_{l-1}-\epsilon_l,\epsilon_l\}=\{\alpha_1,...,\alpha_l\}$.
\end{enumerate}
Let $S=\{s_1,s_2,...,s_{l-1},s_l\}$ be the corresponding set of simple
reflections. Here we consider the case
\[
J=\{s_1,...,s_{l-1}\}.
\]
One checks that $X(J)$ is rationally smooth. An easy calculation, as in Example 4.9
of \cite{R8.5}, yields
\[
P(X(J),t)=\sum_{w\in W^J}t^{2\nu(w)}=\sum_{A\subset \{1,...,l\}}t^{2|A|}=(1+t^2)^l.
\]
\end{example}

Theorem \ref{bettinumber.thm}
is a fundamental ingredient in the calculation of the $H$-polynomal of a simple embedding.
See Theorem \ref{thejewel.thm} below.

\section{Examples}

In this section we calculate the $H$-polynomials of several types of embeddings. First we
discuss all projective, semisimple rank two embeddings.
Then we discuss the wonderful embedding, which was the major motivation for the entire
theory of descent systems and $H$-polynomials.
Finally we discuss rationally smooth, {\bf simple embeddings}. These simple embeddings
illustrate the role of descent systems (\S\ref{descentsys.sec}).

\subsection{Semisimple Rank Two Embeddings}

In this subsection we produce an explicit formula for the
$H$-polynomial of an embedding of type $A_2$, $C_2$ and $G_2$. More details
concerning these examples are written down in \cite{R6}.

Let $G$ be a semisimple group of type $A_2$, $C_2$ or $G_2$, and let $\rho$ be a rational
representation of $G$. As before we define $X_\rho$ to be the $G\times G$-embedding associated
with $\rho$. Then $X_\rho$ is rationally smooth since any two-dimensional torus embedding
is rationally smooth.

Each of these semisimple groups $G$ has dimension $2N+2$ where $N$ is the length of the
longest element in the Weyl group of $G$. Furthermore the Weyl group of $G$ has order
$2N$. One obtains that $N=3$ for type $A_2$, $N=4$ for type $C_2$ and $N=6$
for type $G_2$.

Let $B$ be a Borel subgroup of $G$.
There are three cases to consider here, depending on the closed $G\times G$-orbits.

\begin{enumerate}
\item[(I)]All closed $G\times G$-orbits are isomorphic to $G/B\times G/B$.
\item[(II)] Exactly one closed $G\times G$-orbit is not isomorphic to $G/B\times G/B$.
\item[(III)] Exactly two closed $G\times G$-orbits are not isomorphic to $G/B\times G/B$.
\end{enumerate}

\begin{example} \label{genericssrk2.ex}
All closed $G\times G$-orbits are isomorphic to $G/B\times G/B$.
After a simple calculation, as in Example 6.1 of \cite{R6}, we obtain that the
$H$-polynomial of $X$ is as follows.
\[
H_X(t)=[1+(k-1)t+2k(t^2+\dots+t^N)+(k-1)t^{N+1}+t^{N+2}]H(G/B)
\]
where $H(G/B)=(1+t)(1+t+\dots+t^{N-1})$ and $k$ is the number of closed
$G\times G$-orbits.
\end{example}

\begin{example} \label{nongeneric1ssrk2.ex}
Exactly one closed $G\times G$-orbit is not isomorphic to $G/B\times G/B$.
After a simple calculation, as in Example 6.2 of \cite{R6}, we obtain that
the $H$-polynomial of $X$ is as follows.
\[
H_X(t)=[t^{N+3}+kt^{N+2}+3kt^{N+1}+(4k+1)(t^N+\dots+t^3)+3kt^2+kt+1]H(G/P)
\]
where $H(G/P)=1+t+\dots+t^{N-1}$ and $k$ is the number of closed
$G\times G$-orbits isomorphic to $G/B\times G/B$.
\end{example}

\begin{example} \label{nongeneric2ssrk2.ex}
Exactly two closed $G\times G$-orbits are not isomorphic to $G/B\times G/B$.
After a simple calculation, as in Example 6.3 of \cite{R6}, we obtain that the
$H$-polynomial of $X$ is as follows.
\[
H_X(t)=[t^{N+3}+kt^{N+2}+(3k+1)t^{N+1}+(4k+2)(t^N+\dots+t^3)+(3k+1)t^2+kt+1]H(G/P)
\]
where $H(G/P)=1+t+\dots+t^{N-1}$ and $k$ is the number of closed
$G\times G$-orbits isomorphic to $G/B\times G/B$.
\end{example}

\subsection{The Wonderful Embedding} \label{wonderful.sub}

The {\bf wonderful embedding} corresponds to the case of an embedding $X_\rho$ where the representation $\rho$ is irreducible with highest weight in general position. This corresponds to the situation
of Theorem \ref{almostsmooth.thm} where $J=\emptyset$. We single it out because it has a
special significance in the theory of embeddings \cite{DP}. A semisimple monoid $M$ is
called {\em canonical} if $\Lambda_1=\{e\}$,
and $C_G(e)$ is a maximal torus (this is the smallest the centralizer of an idempotent
can be). These monoids have been studied in detail by Putcha and the author in \cite{PR2}.
Any two canonical monoids $M$ and $M'$ have the same $H$-polynomial since $\mathbb{P}(M)=(M\setminus\{0\})/Z \cong \mathbb{P}(M')=(M'\setminus\{0\})/Z'$
as $G\times G$-varieties. $X$ is related to the
much-studied wonderful embedding of $G/Z$ and we have obtained an explicit cell decomposition
$X=\bigsqcup_r C_r$ of $X$ in \cite{R4}.

\[ H_{\mathbb{P}(M)}(t) = \left[\sum_{u\in W}{t{^{l(w_0)-l(u)+
|I_u|}}}\right]\left[\sum_{v\in W}t^{l(v)}\right] \]
where $I_u=\{s\in S\;|\; u<us \}$. $I_u$ is called the {\em ascent set} of $u$. In the
notation of Theorem \ref{thebiganswer.thm}
\[
\nu(e)=|I_u|
\]
where $e=ue_1u^{-1}$.
Thus the Poincar\'e polynomial of $\mathbb{P}(M)$ is
\[ P(t) = \left[\sum_{u\in W}{t{^{2(l(w_0)-l(u)+
|I_u|)}}}\right]\left[\sum_{v\in W}t^{2l(v)}\right]
\]

\subsection{Simple Embeddings}

In this section we discuss {\bf combinatorially smooth} subsets $J$ of $S$. These subsets
correspond to the rationally smooth $G\times G$-embeddings  $X_\rho$ of a semisimple
group $G$ with $\rho$ an irreducible representation. The formulas of this subsection are the
culmination of the results of \cite{R6,R7,R8,R8.5,R9}.

Recall that
\[
X_\rho = \mathbb{P}(J)
\]
for some unique $J\subseteq S$. In particular, $X_\rho$ depends on $J$, but not $\rho$.
See Definition \ref{combsmooth.def}
and Theorem \ref{almostsmooth.thm}. These embeddings are
{\bf simple embeddings} in the sense that they have exactly one closed $G\times G$-orbit.
We obtain the $H$-polynomial of such $X_\rho$ in terms of the augmented poset
$(E_1,\leq,\{\nu_s\})$. See Definition \ref{augmented.def}. These simple embeddings
correspond to $\mathscr{J}$-irreducible monoids. See \cite{PR1} for a detailed discussion of
$\mathscr{J}$-irreducible monoids. See \cite{R9} for a detailed account of the results of
this section.

The result that gets us going here is the following.

\begin{theorem}  \label{ratsmsimpleemb.thm}
The following are equivalent.
\begin{enumerate}
\item[i)] $J\subseteq S$ is combinatorially smooth.
\item[ii)] $\mathbb{P}(J)$ is rationally smooth.
\end{enumerate}
\end{theorem}

Theorem \ref{ratsmsimpleemb.thm} follows directly from the results of Sections \ref{ratsm.sec}
and \ref{descentsys.sec}.

Let $J\subseteq S$ be combinatorially smooth and let $s\in S\setminus J$, $w\in W^J$.
Then there is at most one irreducible component $C_s\subseteq J$ such that,
for some $t\in J$, $st\neq ts$. Set
\begin{enumerate}
  \item[a)] $\delta(s)=|C_s|+1$, and
	\item[b)] $\nu_s(w)=|\{r\in S^J_s\;|\;w<wr\;\}|$. (which, by a previous definition,
	equals $|A_s^J(w)|$)
\end{enumerate}
Notice that $\delta(s)=1$ if and only if $st=ts$ for all $t\in J$.

Let $w_0\in W^J$ be the longest element (so that $l(w_0)= dim(U_J)$, where $U_J$ is the
unipotent radical of $P_J$.).

The following Theorem is proved in \cite{R9}.

\begin{theorem}                \label{thejewel.thm}
Let $M$ be $\mathscr{J}$-irreducible of type $J\subseteq S$.
Assume that $J\subseteq S$ is combinatorially smooth.
Then the $H$-polynomial of $\mathbb{P}(M)$ is given by
\[
H_{\mathbb{P}(M)}(t)=\left(\sum_{w\in W^J}t^{l(w_0)-l(w)+\nu(w)}\right)\left(\sum_{v\in W^J}t^{l(v)}\right)
\]
where $\nu(w)=\sum_{s\in S\setminus J}\delta(s)\nu_s(w)$ and
$H(J)=\sum_{v\in W^J}t^{l(v)}$, the $H$-polynomial of $G/P_{J}$.
\end{theorem}

\begin{example}   \label{allbutone.ex}
Let $M=M_{n+1}(\mathbb{C})$. Then $M$ is $\mathscr{J}$-irreducible of type $J\subset
S$, where $J=\{s_2,s_3,...,s_n\}$ and $S=S_n=\{s_1,s_2,...,s_n\}\subset W_n$
is of type $A_n$ ($n\geq 1$). In this example
 \begin{enumerate}
   \item[] $S^J=\{s_1,s_2s_1,s_3s_2s_1,...,s_n\cdots s_1\}$, and\\
           $W^J=S^J\sqcup\{1\}$.
 \end{enumerate}
Write $a_i=s_i\cdots s_1$ if $i>1$, and $a_0=1$.
An elementary calculation yields
\begin{enumerate}
\item[] $S\setminus J=\{s_1\}$,\\
        $l(a_i)=i$,\\
        $w_0=s_n\cdots s_1$,\\
        $\delta(s_1)=n$,\\
        $\nu_{s_1}(a_i)=n-i$,\\
        $H(J)=\sum_{i=0}^nt^{2i}$, and \\
        $\mathbb{P}(M)=\mathbb{P}^{(n+1)^2-1}(\mathbb{C})$.
\end{enumerate}
Another elementary calculation (using Theorem~\ref{thejewel.thm}) then yields
\[
H_{\mathbb{P}(M)}(t)=\left(\sum_{i=0}^nt^{(n-i)(n+1)}\right)\left(\sum_{i=0}^nt^i\right)
=\sum_{i=0}^{(n+1)^2-1}t^i.
\]
\end{example}

\begin{example}         \label{allbuttwo.ex}
In this example we illustrate Theorem~\ref{thejewel.thm} by calculating the Poincar\'e
polynomial of $\mathbb{P}(M)$ where $M$ is $\mathscr{J}$-irreducible of type
$J\subset S$, where $S=S_n=\{s_1,s_2,...,s_n\}\subset W_n$ is of type $A_n$ ($n\geq 2$) and
$J=J_n=\{s_3,s_4,...,s_n\}$.

If $w\in W^J_n$ we can write $w=a_pb_q$ where $a_p=s_p\cdots s_1$ ($1\leq p\leq n$)
and $b_q=s_q\cdots s_2$ ($2\leq q\leq n$). We also adopt the peculiar but useful
convention $a_0=1$ and $b_1=1$. Thus
\[
W^J_n=\{a_pb_q\;|\;0\leq p\leq n\;\text{and}\;1\leq q\leq n\}
\]
with uniqueness of decomposition.

Now $S\setminus J=\{s_1,s_2\}$ so that $C_{s_1}=\phi$ and $C_{s_2}=\{s_3,...,s_n\}$.
Thus,
\begin{enumerate}
  \item[i)] $\delta(s_1)=1$, and
  \item[ii)] $\delta(s_2)=(n-2)+1=n-1$.
\end{enumerate}
Then, from Example \ref{thategfromwgspbppp.ex},
\begin{enumerate}
  \item[i)]  $\nu_{s_1}(a_pb_q)=1$ if $p<q$ and\\
             $\nu_{s_1}(a_pb_q)=0$ if $p\geq q$.
  \item[ii)] $\nu_{s_2}(a_pb_q)=n-q$.
\end{enumerate}
Thus, by definition,
\begin{enumerate}
	\item[i)]   $\nu(a_pb_q)=(n-1)(n-q)+1$ if $p<q$ and
	\item[ii)]  $\nu(a_pb_q)=(n-1)(n-q)$ if $p\geq q$.
\end{enumerate}
Finally,
\begin{enumerate}
	\item[i)]  $l(a_pb_q)=p+q-1$, and
	\item[ii)] $a_nb_n\in W^J$ is the longest element.
\end{enumerate}
Thus, for $w=a_pb_q\in W^J$, we obtain by elementary
calculation that
\[
l(w_0)-l(w)+\nu(w)=n-p + n(n-q) + \epsilon
\]
where $\epsilon=1$ if $0\leq p<q\leq n$, and $\epsilon=0$ if $n\geq p\geq q\geq 1$.
Thus
\[
\sum_{w\in W^J}t^{l(w_0)-l(w)+m(w)}=\sum_{0\leq p<q\leq n}t^{n-p + n(n-q) + 1}
+\sum_{n\geq p\geq q\geq 1}t^{n-p + n(n-q)}
\]
The other factor here is
\[
H(J)=\sum_{w\in W^J}t^{l(w)}=\sum_{i=1}^ni(t^{i-1}+t^{2n-i}).
\]
Finally we obtain
\[
H_{\mathbb{P}(M)}(t)=\left(\sum_{0\leq p<q\leq n}t^{n-p + n(n-q) +
1}+\sum_{n\geq p\geq q\geq 1}t^{n-p + n(n-q)}\right)\left(\sum_{i=1}^ni(t^{i-1}+t^{2n-i})\right).
\]
\end{example}

\begin{example}    \label{coolcube.ex}
In this example we consider the root system of type $B_l$. Let $E$ be a real vector
space with orthonormal basis $\{\epsilon_1,...,\epsilon_l\}$. Then
\begin{enumerate}
\item[] $\Phi^+=\{\epsilon_i-\epsilon_j\;|\; i<j\}\cup\{\epsilon_i+\epsilon_j\;|\;
        i\neq j\}\cup\{\epsilon_i\}$, and
\item[] $\Delta =\{\epsilon_1-\epsilon_2,...,\epsilon_{l-1}-\epsilon_l,\epsilon_l\}=\{\alpha_1,...,\alpha_l\}$.
\end{enumerate}
Let $S=\{s_1,s_2,...,s_{l-1},s_l\}$ be the corresponding set of simple
reflections. Here we consider the case the $\mathscr{J}$-irreducible
monoid $M$ of type
\[
J=\{s_1,...,s_{l-1}\}\subseteq S.
\]
We make the following identification.
\[
W^J\cong \{1\leq i_1<i_2<...<i_k\leq l\}
\]
as follows. Given such a sequence, $1\leq i_1<i_2<...<i_k\leq l$, we
define
\[
w(\epsilon_v)=\epsilon_{i_v}\;\text{for}\; 1\leq v\leq k,
\]
and
\[
w(\epsilon_{k+v})=-\epsilon_{j_v}\;\text{for}\; 1\leq v\leq l-k,
\]
where $l\geq j_1>j_2>...>j_{l-k}\geq 1$
(so that $\{1,...,l\}=\{i_1,i_2,...,i_k\}\sqcup\{j_1,j_2,...,j_{l-k}\}$).
One can check that $w\in W^J$
and that, conversely, any element of $W^J$ is of this form.

With these identifications we let $w\in W^J$. We now recall that
\[
A^J(w)=\{r\in S^J\;|\; w<wr\}
\]
and that
\[
S^J=\{s_1\cdots s_l,s_2\cdots s_l,...,s_i\cdots s_l,...,s_{l-1}s_l,s_l\}.
\]
Let $w\in W^J$ correspond, as above, to $i_1<...<i_k$ and
$j_1>...>j_{l-k}$. Let $r_i=s_i\cdots s_l\in S^J$.
By the calculations of \cite{R7}, $w<wr_i$ if and only if $i\leq k$. Thus we obtain
\[
A^J(w)=\{s_1\cdots s_l,...,s_k\cdots s_l\}=\{r\in S^J\;|\; w<wr\}.
\]
Now we can use Theorem~\ref{thejewel.thm} above to obtain the
$H$-polynomial of $M$. Let us first assemble the relevant information.

\begin{enumerate}
 \item $S\setminus J=\{s_l\}$.
 \item $\delta(s_l)=C_{s_l}+1=|\{s_1,...,s_{l-1}\}|+1=l$.
 \item If $w\in W^J$ then $\nu_{s_l}(w)=k$ where
\[
w \leftarrow\rightarrow \{1\leq i_1<i_2<...<i_k\leq l\}
\]
as above.
 \item $\nu(w)=l\nu_{s_l}(w)=kl$.
 \item $l(w_0)-l(w)=\sum_{i\in M'(w)}i$ where $M'(w)=\{i\;|\;
 w(\epsilon_j)=\epsilon_i\;\text{for some}\;j\}=\{i_1,i_2,...,i_k\}$,
 and where $w_0\in W^J$
 is the longest element (notice that $l(w_0)=l(l+1)/2$).
\end{enumerate}
Collecting terms we obtain that,
for $w\in W^J$,
\[
l(w_0)-l(w)+\nu(w)=[\sum_{i\in M'(w)}i]+l|M'(w)|=\sum_{i\in M'(w)}(i+l).
\]
After recalling some elementary generating functions, and applying Theorem~\ref{thejewel.thm},
we obtain that
\[
H_{\mathbb{P}(M)}(t)=\left[\Pi_{k=1}^l(1+t^{k+l})\right]\left[\Pi_{k=1}^l(1+t^k)\right].
\]
The $\Pi_{k=1}^l(1+t^k)$ factor here is $H(G/P_J)=\sum_{v\in W^J}t^{l(v)}$
and the $\Pi_{k=1}^l(1+t^{k+l})$ factor is $\sum_{w\in W^J}t^{l(w_0)-l(w)+\nu(w)}$.
\end{example}

\vspace{20pt}

\noindent Lex E. Renner \\
Department of Mathematics \\
University of Western Ontario \\
London, N6A 5B7, Canada \\

\enddocument